\documentclass[12pt]{amsart} 

\usepackage{amsmath,amssymb,amscd}
\usepackage{newtxtext,newtxmath}
\usepackage[utf8x]{inputenc}

\usepackage{color}
\usepackage{hyperref}
\usepackage[all]{xypic}
\usepackage{enumerate}

\renewcommand{\geq}{\geqslant}

\newcommand{\CH}{\text{CH}}
\newcommand{\CB}{\text{CB}}
\newcommand{\Tor}{\text{Tor}}
\newcommand{\Ext}{\text{Ext}}
\newcommand{\Hom}{\text{Hom}}
\newcommand{\Graph}{\text{\sc Graph}}
\newcommand{\Set}{\text{\sc Set}}

\newcommand{\C}[1]{\mathcal{#1}}
\newcommand{\B}[1]{\mathbb{#1}}
\newcommand{\xla}[1]{\xleftarrow{~#1~}}
\newcommand{\xra}[1]{\xrightarrow{~#1~}}
\newcommand{\lmod}[1]{{#1}\text{-{\bf Mod}}}

\newcommand{\eotimes}[1]{\otimes_{#1}}
\newcommand{\pair}[2]{\text{#1}\atop\text{#2}}

\newtheorem{theorem}{Theorem}[section]
\newtheorem{lemma}[theorem]{Lemma}
\newtheorem{proposition}[theorem]{Proposition}

\theoremstyle{definition}
\newtheorem{definition}[theorem]{Definition}
\newtheorem{remark}[theorem]{Remark}
\newtheorem{example}[theorem]{Example}

\numberwithin{equation}{section}

\usepackage[margin=24mm,top=28mm]{geometry}

\title{Noncommutative fibrations}

\author{Atabey Kaygun}
\email{kaygun@itu.edu.tr}
\address{Department of Mathematics, Istanbul Technical University, Istanbul, Turkey.}

\begin{document}
\begin{abstract}
  We show that faithfully flat smooth extensions of associative unital algebras are
  reduced flat, and therefore, fit into the Jacobi-Zariski exact sequence in Hochschild
  homology and cyclic (co)homology even when the algebras are noncommutative or infinite
  dimensional. We observe that such extensions correspond to étale maps of affine schemes,
  and we propose a definition for generic noncommutative fibrations using distributive
  laws and homological properties of the induction and restriction functors.  Then we show
  that Galois fibrations do produce the right exact sequence in homology.  We then
  demonstrate the versatility of our model on a geometro-combinatorial example.  For a
  connected unramified covering of a connected graph $G'\to G$, we construct a smooth
  Galois fibration $\C{A}_{G}\subseteq\C{A}_{G'}$ and calculate the homology of the
  corresponding local coefficient system.
\end{abstract}
\maketitle

\section*{Introduction}

Based on homological connections between the induction $Ind_{A^e}^{B^e}$ and the
restriction $Res^{B^e}_{A^e}$ functors, in this paper we gather further evidence that the
Hochschild-Jacobi-Zariski exact sequence
\begin{equation}\label{JZ}
  \cdots \to HH_{n+1}(B|A)\to HH_n(A|\Bbbk)\to HH_n(B|\Bbbk)\to HH_n(B|A)\to \cdots
\end{equation}
is the long exact sequence of associated to a fibration of ordinary (noncommutative)
affine spaces $Spec(B)\to Spec(A)$ when the extension $A\subseteq B$ is reduced flat.
First, we prove in Theorem~\ref{FlatRelative} that for faithfully flat extensions reduced
flatness is equivalent to $A^e$-flatness of $\Sigma_{B|A}$ the kernel of the relative multiplication map $B\otimes_A B\to B$.
Then in Theorem~\ref{AlmostSmoothBegetsReducedFlat} we obtain a
\emph{faithfully flat étale descent} result analogous to~\cite[Theorem
(0.1)]{WeibelGeller:EtaleDescent} but for all associative unital algebras not just
commutative ones: We show that any faithfully flat smooth extension $A\subseteq B$ is
reduced flat, and therefore, the geometric fibre of $Spec(B)\to Spec(A)$ is homologically
trivial.  The result follows from the fact that now we have the Hochschild-Jacobi-Zariski
exact sequence with coefficients~\eqref{HomologicalHJZ} for faithfully flat smooth
extensions, and the fact that the restriction functor already induces the correct
isomorphisms in homology by Proposition~\ref{AlmostSmoothness}.

There is an analogous Jacobi-Zariski exact sequence for extensions of commutative algebras
in André-Quillen (co)homology without any further restriction on the
extension~\cite{Quillen:AQCohomology, Andre:AQCohomology}.  However, our~\eqref{JZ} is
exact for commutative and noncommutative algebras alike even when they are not finite
dimensional or essentially of finite type.  The results in this paper came from an
observation that smooth extensions and reduced flat extensions are related in terms of
homological properties of their induction and restriction functors: while the
multiplication map $Ind_{A^e}^{B^e}A\to B$ induces a Hochschild cohomological equivalence
in degrees higher than 1 for a smooth extension, for a reduced flat extension one gets a
Hochschild homological equivalence for the same range for the natural $A$-bimodule
embedding $A\to Res^{B^e}_{A^e}B$.  We refer the reader to Subsection~\ref{Relations} for
a detailed analysis of these connections.

Based on the results we obtained in Section~\ref{sect:ReducedFlatSmooth}, we
propose that a special class of extensions of unital associative algebras that
contains the class of Hopf-Galois
extensions~\cite{Schneider:PrincipalHomogeneousSpaces} constitutes an appropriate
model for generic smooth noncommutative fibrations.  We define a noncommutative
(unramified) fibration as a flat extension $A\subseteq B$ that admits a (bijective)
distributive law~\cite{Beck:DistributiveLaws}
$\bowtie\colon C\otimes B\to B\otimes C$ together with an epimorphism of
$B$-bimodules $can\colon Ind_{A^e}^{B^e}A\to B\bowtie C$ that satisfies an
invariance condition $A=B^C$. Then in Theorems~\ref{MainHJZ} and ~\ref{MainCJZ} we
get the correct fibration sequence with the appropriate fibre for the map
$Spec(B)\to Spec(A)$ for Galois fibrations.  Since we formulate our extensions in
terms of distributive laws instead of cleft Hopf-Galois extensions, the extensions
we consider model generic fibrations, not just principal fibrations.  We refer the
reader to Section~\ref{sect:Fibrations} for details.

We demonstrate the versatility of our model on a geometro-combinatorial example.
For a connected unramified covering $G'\to G$ of a connected graph $G$, we
construct an unramified reduced flat extension $\C{A}_{G}\subseteq\C{A}_{G'}$ of
noncommutative algebras in Subsection~\ref{HochschildOfGraphAlgebras}.  We then
show in Theorem~\ref{UnramifiedCovering} that for such extensions, we get the right
analogue of the long exact sequence of a fibration in cyclic homology.  Then we
extend our result to local coefficient systems on graphs and their cohomology in
Theorem~\ref{LocalCoefficients}.

The particular result we obtain in Theorem~\ref{LocalCoefficients}, combined with
Burghelea's~\cite{Burghelea:CyclicHomologyOfGroupRings}, is consistent
with~\cite[Chapter III, Theorem 2.20]{Milne:EtaleCohomology} and~\cite[Example
2.2]{WeibelGeller:EtaleDescent} where one obtains the homology of a Galois
coverings of schemes from a Hochschild-Serre hyper-homology spectral sequence in
which they combine the group cohomology of the structure group of the fibration and
the homology of the base.  This consistency indicates that our proposal is sound
geometrically.  Since Theorem~\ref{LocalCoefficients} is a direct consequence of
Theorem~\ref{MainCJZ}, we also see that for cleft Hopf-Galois extensions the
Hochschild homology of such an extension relative to the base is the homology of
the underlying Hopf algebra.  Hence our proposal is sound algebraically as well.

\subsection*{Plan of the article}

We recall the results we need on reduced flat and smooth extensions in
Section~\ref{sect:ReducedFlatSmooth}.  Our Proposition~\ref{AlmostSmoothness} and
Proposition~\ref{InductionEquivalence} identify the reduced flat extensions and
smooth extensions in terms of homological conditions on the induction and
restriction functors.  Then we define unramified and Galois fibrations, and discuss
connections between various types of extensions and fibrations in
Section~\ref{sect:Fibrations}.  In Subsection~\ref{subsect:Engine} we prove our
main technical results.  First, we show that the relative Hochschild homology of a
Galois fibration yields the correct homology of the fibre in Theorem~\ref{MainHJZ}.
Then in Theorem~\ref{MainCJZ}, we show that for reduced flat Galois fibration, we
have the required long exact sequences in Hochschild homology and cyclic
(co)homology.  We apply our main results to graph extension algebras in
Section~\ref{sect:GraphCovering}.  In Subsection~\ref{subsect:LocalCoefficients},
we define local coefficient systems on graphs, and finally in
Theorem~\ref{LocalCoefficients} we prove that the relative homology of a
noncommutative fibration with coefficients in a local system gives us the group
homology of the local coefficients.

\subsection*{Notation and conventions}

Throughout this article, we are going assume $\Bbbk$ is a ground field of
characteristic 0.  All unadorned tensor products $\otimes$ are taken over $\Bbbk$.
All algebras are assumed to be over $\Bbbk$, and all are unital and associative.
However, they need not be commutative or finite dimensional.  We use $\Sigma_{B|A}$ to denote the kernel of the relative multiplication map $B\otimes_A B\to B$ for an algebra extension $A\subseteq B$.  All modules are
assumed to be left modules unless otherwise stated.  We use $\lmod{A}$ to denote a
small category of $A$-modules.  For an algebra $A$, we use $A^e$ to denote the
enveloping algebra $A\otimes A^{op}$.  Thus modules over $A^e$ are exactly
bimodules over $A$.  We use the homological convention for complexes: all complexes
are positively graded and differentials reduce the degree by one.  We use $\Tor^A$
and $\Ext_A$ to denote the derived bifunctors of the tensor product $\eotimes{A}$
and $\Hom_A$-bifunctors, respectively.  We are going to use $\CB_*$ to denote the
bar complex, and $\CH_*$ to denote the Hochschild complex.  Also, we use $HH_*$ for
the Hochschild homology, and $HC_*$ for the cyclic homology functors.  All graphs
are assumed to be undirected and simple, but they need not be finite.  In
particular, we have no loops on a vertex, and no multiple edges between any two
vertices.

\subsection*{Acknowledgments}

This work is partially supported by NCN grant UMO-2015/19/B/ST1/03098.

\section{Reduced flat and almost smooth extensions}\label{sect:ReducedFlatSmooth}

For this section, we assume we have an extension of unital associative algebras
$A\subseteq B$ such that $B$ viewed as a left and right $A$-module is flat.

\subsection{Relative Hochschild (co)homology}

Given an extension $A\subseteq B$, the relative two sided bar complex $\CB_*(B|A)$
is defined to be the graded $B^e$-module
\[ \CB_n(B|A) = \underbrace{B\eotimes{A}\cdots\eotimes{A} B}_\text{$n+2$-times} \]
For every $n\geq 1$, the differentials $d_n\colon \CB_n(B|A)\to \CB_{n-1}(B|A)$
are defined as
\[ d_n(u_0\otimes\cdots\otimes u_{n+1}) = \sum_{i=0}^n (-1)^i (\cdots\otimes
  u_{i-1}\otimes u_iu_{i+1}\otimes u_{i+2}\otimes\cdots) \] for every homogeneous
tensor $u_0\otimes\cdots\otimes u_{n+1}\in \CB_n(B|A)$, then extended linearly.
Since $\CB_*(B|A)$ is a $(B^e,A^e)$-projective resolution of $B$ as a
$B^e$-module, for any $B^e$-module $X$ we write the relative Hochschild chain and
cochain complexes as
\[ \CH_*(B|A,X) := \CB_*(B|A)\eotimes{B^e} X \text{ and } \CH^*(B|A,X) :=
  \Hom_{B^e}(\CB_*(B|A),X) \] that yield the relevant relative Hochschild homology
and cohomology groups $HH_*(B|A,X)$ and $HH^*(B|A,X)$, respectively.  In the case
$A=\Bbbk$, we simply write $\CB_*(B)$ and $HH_*(B)$ instead of $\CB_*(B|\Bbbk)$ and
$HH_*(B|\Bbbk)$.

\subsection{Induction and restriction}

We have two related functors:
\begin{enumerate}[(i)]
\item Induction $Ind_{A^e}^{B^e}X := B\eotimes{A} X\eotimes{A} B$, and
\item Restriction $Res^{B^e}_{A^e}Y$ where we view $Y$ as an $A$-bimodule via the
  inclusion $A\subseteq B$
\end{enumerate}
for every $X\in \lmod{A^e}$ and $Y\in\lmod{B^e}$.

\begin{lemma}\label{PrimaryLemma}
  For every $X\in\lmod{A^e}$ we have
  \[ \Tor^{A^e}_n(Res^{B^e}_{A^e}B,X) \cong \Tor^{B^e}_n(B,Ind^{B^e}_{A^e}X) \] for
  every $n\geq 0$.
\end{lemma}

\begin{proof}
  We observe that the $\CB_*(B)$ is a free resolution of the $B$-bimodule $B$.
  Since we assumed $B$ is a right and left flat $A$-module, we also have that
  $\CB_*(B)$ is a flat resolution of the $A$-bimodule $B$.  Then
  \begin{align*}
    \Tor^{A^e}_n(Res^{B^e}_{A^e}B,X)
    \cong & H_n \left(\CB_*(B)\eotimes{A^e} X\right) \\
    = & H_n \left(\CB_*(B)\eotimes{B^e} (B\eotimes{A} X\eotimes{A} B)\right)\\
    \cong & \Tor^{A^e}_n(B,Ind^{B^e}_{A^e}X)
  \end{align*}
  as we wanted to prove.
\end{proof}

\begin{remark}
  By Lemma~\ref{PrimaryLemma} we have a sequence of natural morphisms
  \begin{align*}
    HH_n(A,Y) = \Tor^{A^e}_n(A,Y) & \xra{\xi_Y}\\
    \Tor^{A^e}_n & (Res^{B^e}_{A^e}B,Y)
    \cong \Tor^{B^e}_n(B,Ind^{B^e}_{A^e}Y) = HH_n(B,Ind^{B^e}_{A^e}Y)
  \end{align*}
  and
  \begin{align*}
    HH_n(A,Res^{B^e}_{A^e}X) = \Tor^{A^e}_n(A,Res^{B^e}_{A^e}X)
    \cong \Tor^{B^e}_n & (Ind^{B^e}_{A^e}A,X)\\
    & \xra{\nu_X}  \Tor^{B^e}_n(B,X) = HH_n(B,X) 
  \end{align*}
  for every $X\in \lmod{B^e}$, $Y\in \lmod{A^e}$ and $n\geq 0$.  In the following
  subsections, we are going to show that $\xi$ viewed as a natural transformation
  of functors is an isomorphism when the extension is reduced flat, and $\nu$ again
  viewed as a natural transformation of functors is an isomorphism when the
  extension is (almost) smooth, both for a certain range of $n$.  Moreover, we are
  also going to show that when the extension is faithfully flat then the fact that
  $\nu$ is an isomorphism implies so is $\xi$.
\end{remark}

\subsection{Almost smooth extensions}\label{AlmostSmooth}

For an extension of $\Bbbk$-algebras $A\subseteq B$, we define $\Sigma_{B|A}$ to
be the kernel of the relative multiplication map
$Ind_{A^e}^{B^e}A = B\eotimes{A} B\to B$ as a morphism of $B^e$-modules.  We call a
flat extension \emph{(almost) smooth} if $\Sigma_{B|A}$ is a projective
(resp. flat) $B^e$-module~\cite{Schelter:SmoothAlgebras}.  Notice that when an
extension is smooth then it is also almost smooth.

\begin{proposition}\label{AlmostSmoothness}
  A flat extension $A\subseteq B$ is almost smooth if and only if we have
  $HH_n(B,X)\cong HH_n(A,Res_{A^e}^{B^e}X)$ for every $X$ and for every $n\geq 2$.
\end{proposition}

\begin{proof}
  The proof follows from the fact that $A\subseteq B$ is almost smooth
  if and only if we have a sequence of isomorphisms of the form
  \begin{align*}
    HH_n(A,Res^{B^e}_{A^e}X) = \Tor^{A^e}_n(A,Res^{B^e}_{A^e}X) \cong & \Tor^{B^e}_n(Ind_{A^e}^{B^e}A,X)\\
    \cong & \Tor^{B^e}_n(B,X) = HH_n(B,X) 
  \end{align*}
  for every $n\geq 2$ and $X\in\lmod{B^e}$.  
\end{proof}

\begin{remark}
  There is a version of Proposition~\ref{AlmostSmoothness} for smooth extensions
  that works with Hochschild cohomology instead of homology that says
  $A\subseteq B$ is smooth if and only if
  \[ HH^n(B,X) \cong HH^n(A,Res^{B^e}_{A^e} X) \] for every
  $X\in\lmod{B^e}$ and $n\geq 1$.
\end{remark}

\subsection{Reduced flat extensions}

We now recall from~\cite{Kaygun:JacobiZariski} that we call an extension
$A\subseteq B$ as \emph{reduced flat} when the cokernel $B/A$ of the $A$-bimodule
inclusion $A\to Res^{B^e}_{A^e}B$ is flat as a $A$-bimodule.  We also observe that
reduced flatness of the extension is equivalent to the fact that the Hochschild
homology of $H_n(A,X)$ of $A$ with coefficients in any $X\in\lmod{A^e}$, and the
torsion groups $\Tor^{A^e}_n(Res^{B^e}_{A^e}B,X)$ are isomorphic for all $n\geq 1$.
Combining this result with Lemma~\ref{PrimaryLemma} we get

\begin{proposition}\label{InductionEquivalence}
  A flat extension $A\subseteq B$ is reduced flat if and only if we have natural
  isomorphisms of the form $HH_n(A,X)\cong HH_n(B,Ind_{A^e}^{B^e}X)$ for every
  $X\in\lmod{A^e}$ and for every $n\geq 1$.
\end{proposition}

We will say that an extension $A\subseteq B$ \emph{satisfies
  Hochschild-Jacobi-Zariski (resp. cyclic-Jacobi-Zariski)
  condition}~\cite[3.5.5.1]{Loday:CyclicHomology} if we have a long exact sequence
in Hochschild homology (resp. cyclic homology) of the form
\begin{equation}
  \label{HomologicalHJZ}
  HH_{n+1}(B|A,X)\to HH_n(A,Res^{B^e}_{A^e}X)\to HH_n(B,X)\to HH_n(B|A,X)
\end{equation}
for every $n\geq 1$, and for every $X\in\lmod{B^e}$.

\begin{proposition}[{\cite[Theorem 4.1 and Theorem 4.2]{Kaygun:JacobiZariski}}]\label{HJZ}
  If a flat extension $A\subseteq B$ is reduced flat then the extension satisfies
  both Hochschild-Jacobi-Zariski and cyclic-Jacobi-Zariski conditions for every
  $X\in\lmod{B^e}$.
\end{proposition}

\begin{remark}
  Recall that the relative homology groups $HH_n(B|A,X)$ measure the failure of
  the extension of being smooth since we have both the exact sequence by
  Proposition~\ref{AlmostSmoothness} and Proposition~\ref{HJZ} for a almost smooth
  reduced flat extensions $A\subseteq B$.  This is rather subtle: almost smoothness
  does imply relative homology vanishes since absolute $B^e$-flatness of
  $\Sigma_{B|A}$ implies that its relative $(B,A)$-flatness as a $B^e$-module.
  However, the converse need not be true in general.  The fact that the relative
  homology vanishes for $n\geq 2$ implies $\Sigma_{B|A}$ is only $B^e$-flat
  \emph{relative to} $A$. The fact that $B$ is reduced flat over $A$ gives
  us~\eqref{HomologicalHJZ}, and then we get the isomorphisms
  $HH_n(A,Res^{B^e}_{A^e}X)\to HH_n(B,X)$ for the required range, and then
  Proposition~\ref{AlmostSmoothness} gives us the absolute $B^e$-flatness.
\end{remark}

\subsection{Faithfully flat almost smooth extensions}

\begin{theorem}\label{FlatRelative}
  Assume $B$ is faithfully flat over $A$.  Then $B$ is reduced flat over $A$ if and
  only if $Res^{B^e}_{A^e}\Sigma_{B|A}$ is a flat $A$-bimodule.
\end{theorem}

\begin{proof}
  We start by observing that there is a natural isomorphism of $A^e$-modules of the
  form
  $Res^{B^e}_{A^e}\Sigma_{B|A} \cong (Res^{B^e}_{A^e}B/A) \eotimes{A}
  Res^{B^e}_{A^e} B$ coming from the diagram
  \[\xymatrix{
    &  & A\eotimes{A} Res^{B^e}_{A^e} B \ar[r]^{\cong} \ar[d] & Res^{B^e}_{A^e} B \ar@{=}[d] & \\
    0\ar[r] & Res^{B^e}_{A^e}\Sigma_{B|A}\ar[r] & Res^{B^e}_{A^e} (B\eotimes{A} B) \ar[r] & Res^{B^e}_{A^e} B \ar[r] & 0
  }\]
  using the Snake's Lemma.  Let us drop the use of $Res^{B^e}_{A^e}$ to simplify the
  notation.  Then we see that the functor
  $(\ \cdot\ )\eotimes{A^e} (B/A\eotimes{A} B)$ is exact if and only if
  \[ (\ \cdot\ )\eotimes{A^e} (B/A\eotimes{A} B) \cong
    B/A\eotimes{A^e}(B\eotimes{A} \ \cdot\ ) \] is exact.  Since we assumed $B$ is
  faithfully flat over $A$, the flatness of $B/A$ of $A^e$-module is equivalent to
  the flatness of $\Sigma_{B|A}$ as a $A^e$-module.
\end{proof}

\begin{remark}
  One should think of Theorem~\ref{FlatRelative} as a \emph{faithfully flat descent}
  result because the fact that an extension $A\subseteq B$ is reduced flat is equivalent
  to the fact that the Amitsur complex
  \[ A\to B \to B\eotimes{A}B\to B\eotimes{A}B\eotimes{A}B \to \cdots \] is
  exact~\cite{RosenbergZelinsky:AmitsurComplex}, or that the cobar complex of the
  Sweedler coring~\cite[Chapter 4, Section 25]{BrzezinskiWisbauer:ComodulesCorings}
  is contractible.  In fact, the \emph{descent data} for an
  extension~\cite{KnusOjanguren:Descent, Grothendieck:DescentI} is a specific
  prescription for a contracting homotopy for the Amitsur complex.
\end{remark}

\begin{theorem}\label{AlmostSmoothBegetsReducedFlat}
  Every faithfully flat almost smooth extension $A\subseteq B$ is reduced
  flat. Then there are isomorphisms in Hochschild homology $HH_n(A)\cong HH_n(B)$
  for every $n\geq 2$, and in cyclic homology $HC_{n+2}(B|A)\cong HC_n(B|A)$ for
  every $n\geq 1$.
\end{theorem}

\begin{proof}
  Let us first prove that $B$ is reduced flat over $A$ when the extension is almost
  smooth.  Since $A\subseteq B$ is faithfully flat almost smooth, we have that
  $\Sigma_{B|A}$ is $B^e$-flat.  By Lazard's Theorem~\cite{Lazard:Flatness},
  $\Sigma_{B|A}$ is a flat $B$-bimodule if and only if it is a filtered colimit of free
  $B$-bimodules.  But free $B$-bimodules are a subclass of flat $A$-bimodules since our
  extension $A\subseteq B$ is flat, and every filtered colimit of flat $A$-bimodules is
  also flat.  Thus $Res^{B^e}_{A^e}\Sigma_{B|A}$ is a flat $A$-bimodule since it was a
  flat $B^e$-module.  Since $B$ is faithfully flat over $A$, this is equivalent to $B$
  being reduced flat over $A$ by Theorem~\ref{FlatRelative}. For the second assertion we
  observe that we have a sequence of isomorphisms
  \[ HH_n(A) \xra{\cong} HH_n(A,Res^{B^e}_{A^e}B) \cong
    HH_n(B,Ind_{A^e}^{B^e}A) \xra{\cong} HH_n(B)
  \]
  which proves we have the desired isomorphisms for $n\geq 2$.  The last assertion
  follows from Connes' SBI-sequence and the fact that $HH_n(B|A)$ is trivial for
  $n\geq 2$. See~\cite[Chapter 3]{Connes:Book} or ~\cite[Chapter II, Section
  2.2]{Loday:CyclicHomology}.
\end{proof}

\begin{example}
  Let $A$ be an algebra with Hochschild homological dimension 0. Such algebras are also
  known as \emph{amenable}~\cite{Johnson:CohomologyInBanachAlgebras} in the continuous
  Hochschild homology context.  In that case every $A^e$-module is also $A^e$-flat,
  i.e. all extensions of $A$ are almost smooth.  The typical examples are field extensions
  $k\subseteq A$, group algebras $k[G]$ over a finite group $G$, or algebra of functions
  $k(G)$ on a compact group $G$.  Then $\Sigma_{B|A}$ is automatically $A^e$-flat, and as
  long as $B$ is faithfully flat over $A$ the extension is also reduced flat by
  Corollary~\ref{AlmostSmoothBegetsReducedFlat}.  So, all faithfully flat extensions over
  amenable algebras are reduced flat.
\end{example}

\section{Fibrations of Algebras}\label{sect:Fibrations}

\subsection{Transpositions, distributive laws and fibrations}\label{Galois}

Let $C$ be a unital associative $\Bbbk$-algebra and let $B$ be an ordinary
$\Bbbk$-vector space.  A morphism of $\Bbbk$-vector spaces
$\omega\colon C\otimes B\to B\otimes C$ is called a \emph{left
  transposition}~\cite{Kaygun:UniversalHopfCyclicTheory} if the following diagram
is commutative
\begin{equation}
  \xymatrix{
    C\otimes C\otimes B \ar[r]^{C\otimes\omega}\ar[d]_{\mu_C\otimes B}
    & C\otimes B\otimes C \ar[r]^{\omega\otimes C}
    & B\otimes C\otimes C \ar[d]^{B\otimes \mu_C}\\
    C\otimes B \ar[rr]^\omega & &   B\otimes C
  }
  \qquad
  \xymatrix{
    & B\ar[dr]^{B\otimes 1}\ar[dl]_{1\otimes B}\\
    C\otimes B \ar[rr]^{\omega} & & B\otimes C
  }
\end{equation}
Right transpositions are defined similarly so that the inverse of a left
transposition, should it exist, would be a right transposition, and vice versa.

Now, assume $B$ and $C$ are two unital associative $\Bbbk$-algebras.  A morphism
of $\Bbbk$-vector spaces $\bowtie\colon C\otimes B\to B\otimes C$ is called a
\emph{a distributive law}~\cite{Beck:DistributiveLaws} if $\bowtie$ is a left
transposition with respect to $C$ and a right transposition with respect to $B$.

One can easily show that $\bowtie$ is a distributive law if and only if
$(\mu_B\otimes\mu_C)(B\otimes\bowtie\otimes C)$ is an associative unital product on
$B\otimes C$ whose unit is $1_B\otimes 1_C$.  We are going to use $B\bowtie C$ to denote
this algebra.  The triple $(B,C,\bowtie)$ is also called \emph{a matched pair of algebras}
and also \emph{twisted tensor product of algebras}~\cite{CapEtAl:TwistedTensorProduct}.

A right transposition $\omega\colon C\otimes B\to B\otimes C$ is called
\emph{invariant} if $B$ has a set of algebra generators $X$ such that for every
$x\in X$ and $c\in C$ there is another $c'\in C$ such that
$\omega(c\otimes x) = x\otimes c'$.  Invariant left transpositions are defined
similarly.

Let $\omega\colon C\otimes B\to B\otimes C$ be a right transposition.  A subalgebra
$A\subseteq B$ is called \emph{$C$-invariant} if the transposition $\omega$
restricts to an invariant transposition on $A$.  The largest subalgebra of $B$
which is $C$-invariant with respect to a transposition is denoted by $B^C$.

An extension of algebras $A\subseteq B$ is called \emph{a fibration} with fibres in
an algebra $C$ if
\begin{enumerate}[(i)]
\item there is an distributive law
  $\bowtie\colon C\otimes B\to B\otimes C$ with $A\subseteq B^C$, and
\item there is an epimorphism of $B^e$-modules
  $can\colon B\eotimes{A} B \xra{} B\bowtie C$.
\end{enumerate}
We call a fibration \emph{unramified} (resp. étale) when $\bowtie$ is invertible
(resp. injective.)  We call a fibration \emph{smooth} if the kernel of $can$ is
$B^e$-projective.  A fibration is called \emph{separable} if $can$ is a split epimorphism
of $B^e$-modules.  We call a fibration $A\subseteq B$ with fibres in $C$ as \emph{a Galois
  fibration} when the canonical map $can$ is an isomorphism of $B^e$-modules.

We would like to emphasize that any fibration $A\subseteq B$ with fibres in $C$ (be it unramified, étale, separable, or smooth) presupposes a distributive law $\bowtie\colon C\otimes B\to B\otimes C$ with $A\subseteq B^C$, and an epimorphism of $B^e$-modules $can\colon B\eotimes{A} B \xra{} B\bowtie C$.

\subsection{Examples}

\begin{example}
  Let $H$ be a Hopf algebra and let $B$ be a $H$-comodule algebra. In other words, $H$
  coacts on $B$ via a coaction $\lambda\colon B\to H\otimes B$  such that
  \[ (ab)_{(-1)}\otimes (ab)_{(0)} = a_{(-1)}b_{(-1)}\otimes a_{(0)}b_{(0)} \] for every
  $a,b\in A$ where we use the notation
  \[ \lambda(b) = b_{(-1)}\otimes b_{(0)} \] for every $b\in B$.  Let $A:=B^H$ where
  \[ B^H = \{b\in B\mid \lambda(b) = b\otimes 1_H \} \] Then there is an invertible
  distributive law $\bowtie \colon H\otimes B\to B\otimes H$ by
  \[ \bowtie(h\otimes b) = b_{(0)}\otimes hb_{(-1)} \] for every $h\in H$ and $b\in B$ and
  $A\subset B^H$ with respect to this distributive law and there is an $B^e$-bimodule
  isomorphism $can\colon B\otimes_A B\to B\otimes H$ when $A\subseteq B$ is a Hopf-Galois
  extension over $H$~\cite{Schneider:PrincipalHomogeneousSpaces}. 
\end{example}

\begin{example}
  Let $L$ be an algebra, and $B$ be an Ore extension $L[X,\alpha,\delta]$ over $L$ whose
  derivation part is trivial, i.e. $\delta=0$.  Let us set $B = L[X,\alpha,0]$ and
  $A=k[X]$, and we consider the extension $A\subseteq B$.  Notice that since $B$ has a
  basis in monomials of the form $u X^n$ where $u\in L$ and $n\in\B{N}$, we have a
  $k$-vector space isomorphism $B\cong L\otimes A$.  Then we see that we get a Galois
  fibration $A\subseteq B$ with fibres in $L$ since there is a natural isomorphism of
  $B$-bimodules of the form
  \[ B\otimes_A B \cong L\otimes B \] with the following invertible distributive law
  $\bowtie \colon B\otimes L\to L\otimes B$
  \[ \bowtie(uX^m\otimes v) = u\alpha^m(v)\otimes X^n \] for every $u,v\in L$ and monomial
  $uX^n\in B$.  Notice that $B^L = A$ and that $B\cong L\bowtie A$. Thus every Ore
  extension of the form $L[X,\alpha,0]$ is a Galois fibration $A\subseteq L[X,\alpha,0]$
  with fibres in $L$.

  When we consider the relative multiplication map $B\otimes_A B\to B$ we see that it
  reduces to $\mu_L\otimes id_A\colon L\otimes L\otimes A\to L\otimes A$ the multiplication
  map on $L$ tensored with the identity on $A$.  So, the kernel of the relative multiplication map 
  $\Sigma_{B|A}$ is isomorphic $\Sigma_{L|k}\otimes A$.  Since $A$ is invariant, as a
  $A$-bimodule this is just direct sum of $\dim_k\Sigma_{L|k}$-many copies of $A$.  Since
  $B$ is already faithfully flat over $A$, in order for this extension to be reduced flat
  we need $A$ to be $A^e$-flat as well, which we know is not the case.  In other words,
  Ore extensions are Galois fibrations but are not reduced flat extensions.
\end{example}

\begin{example}
  Let $G$ be a group acting on an algebra $L$ via algebra automorphisms
  $\triangleright \colon k[G]\otimes L\to L$.  We define a distributive law
  $\bowtie \colon k[G]\otimes L\to L\otimes k[G]$ by
  \[ \bowtie(g\otimes u) = g\triangleright u\otimes g \] for every $g\in G$ and $u\in L$.
  We let $B:= L\bowtie k[G]$ and $A=k[G]$.  As in the case with Ore extension,
  $B\cong L\otimes A$ as $k$-vector spaces and
  \[ B\otimes_A B \cong L\otimes B \] and with a distributive law
  $\bowtie\colon B\otimes L\to L\otimes B$ defined as
  \[ \bowtie(ug\otimes v) = u(g\triangleright v)\otimes g \] for every $ug\in B$ and
  $v\in L$.  We see that $B^L = A$ and we have a Galois fibration $A\subseteq B$ with
  fibres in $L$.  One can think of this as the more general version of the Ore extension
  we gave above where $G=\B{N}$.

  As for reduced flatness, the same argument above that worked for Ore extensions works
  here too: the kernel of the relative multiplication map
  $\Sigma_{B|A}$ is isomorphic to
  $\Sigma_{L|k}\otimes A$ which is again $\dim_k\Sigma_{L|k}$ copies of $A$.  So, since
  $B$ is already faithfully flat over $A$, for this extension to be reduced flat we need
  $A$ to be $A^e$-flat.  Since $A$ is the group algebra $k[G]$, this means the group has
  to have trivial homology.  In this case any finite group $G$ would work since we assume
  $char(k)$ is $0$.

  In short, all extensions of the form $k[G]\subseteq L\bowtie k[G]$ are reduced flat
  Galois fibrations if $G$ is finite and $char(k)=0$.
\end{example}

\subsection{Relationships}\label{Relations}

Assume $A\subseteq B$ is a flat extension of algebras.  Let us depict the
relationships between various objects we defined in this paper as follows:
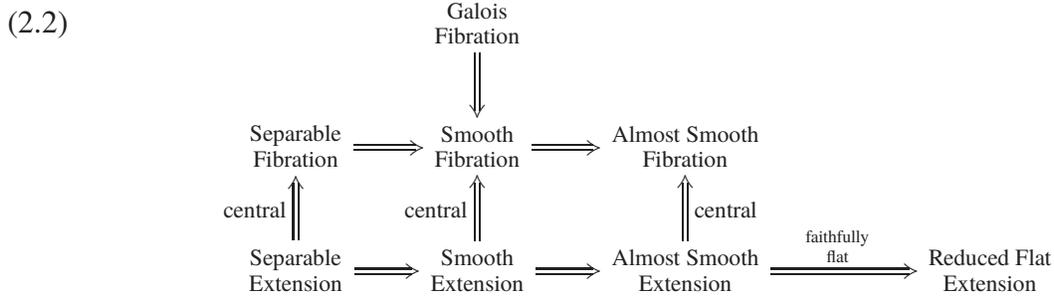
\begin{equation}
  \xymatrix{
    & \pair{Galois}{Fibration}\ar@{=>}[d]\\
    \pair{Separable}{Fibration} \ar@{=>}[r] & \pair{Smooth}{Fibration} \ar@{=>}[r]
    & \pair{Almost Smooth}{Fibration} \\
    \pair{Separable}{Extension} \ar@{=>}[u]^{\text{central}} \ar@{=>}[r] 
    & \pair{Smooth}{Extension} \ar@{=>}[u]^{\text{central}} \ar@{=>}[r] 
    & \pair{Almost Smooth}{Extension} \ar@{=>}[u]_{\text{central}}\ar@{=>}[rr]^{\pair{faithfully}{flat}}
    & & \pair{Reduced Flat}{Extension}
  }
\end{equation}

Let us define $\Sigma_{B|A}$ as $ker(Ind^{B^e}_{A^e}A\to A)$, and $\mho^{B|A}$ as
$coker(A\to Res^{B^e}_{A^e}B)$, and consider the short exact sequences
\[ 0 \to \Sigma_{B|A}\to Ind_{A^e}^{B^e}(A) \to B \to 0 
   \quad\text{ and }\quad
   0 \to A \to Res^{B^e}_{A^e}(B)\to \mho^{B|A} \to 0 
\]
So, we call an embedding $A\subseteq B$ of unital associative algebras as a
\begin{enumerate}[(a)]
\item \emph{Reduced flat extension} when $\mho^{B|A}$ is a flat $A$-bimodule.
  Reduced flatness is equivalent to $Res^{B^e}_{A^e}\Sigma_{B|A}$ being a flat
  $A$-bimodule by Theorem~\ref{FlatRelative} when $B$ is faithfully flat over
  $A$.
\item \emph{Smooth extension} when $\Sigma_{B|A}$ is a projective $B$-bimodule.
  If in addition $A$ is central in $B$, by choosing the trivial distributive law
  $B\otimes \Bbbk\to \Bbbk\otimes B$ we can form a smooth fibration.
\item \emph{Almost smooth extension} when $\Sigma_{B|A}$ is a flat $B$-bimodule.
  Notice that when $\Sigma_{B|A}$ is projective, then it is also flat which means
  smoothness always implies almost smoothness.  Since every finitely presented flat
  $B^e$-module is projective, in cases where $B^e$ is noetherian almost smoothness
  and smoothness agree.
\item \emph{Smooth fibration} if the kernel of the structure map $can$ is
  $B^e$-projective, and \emph{almost smooth} if the kernel is flat as a
  $B^e$-module as expected.
\item \emph{Separable extension}~\cite[1.2.12]{Loday:CyclicHomology} when $B$ is a
  projective $B^e$-module relative to $A$ which implies $\Sigma_{B|A}$ is
  projective relative to $A$.  This means the extension is smooth relative to $A$.
\item \emph{Separable fibration} when $B\bowtie C$ is a projective $B^e$-module
  relative to $A$.  All separable extensions are separable fibrations over base
  algebra $A$.
\item \emph{Galois fibration} if the structure map $can$ is an isomorphism.  Then
  one also gets that it is a smooth fibration for free.  All Hopf-Galois extensions
  yield Galois fibrations by
  definition~\cite{Schneider:PrincipalHomogeneousSpaces}.  In particular, all
  Hopf-Galois extensions are unramified Galois fibrations.
\end{enumerate}

\subsection{Homology of Galois fibrations}\label{subsect:Engine}

\begin{theorem}\label{MainHJZ}
  Assume $A\subseteq B$ is a Galois fibration with fibres in $C$.  Then we have
  natural isomorphisms in Hochschild homology and cyclic (co)homology of the form
  (written here only for Hochschild homology)
  $$ HH_n(B|A,X)  \cong HH_n(C,X/[X,A]) $$
  for every $X\in\lmod{B^e}$ and for all $n\geq 0$.
\end{theorem}

\begin{proof}
  In order to calculate $HH_n(B|A,X)$ we are going to use
  $X\eotimes{B^e}\CB_*(B|A)$:
  \begin{align*}
    \CH_n(B|A,X)
    \cong & A\eotimes{A^e}(X\eotimes{A}\underbrace{B\eotimes{A}\cdots\eotimes{A}B}_{\text{$n$-times}})\\
    \cong & A\eotimes{A^e}(X\eotimes{B}\underbrace{B\eotimes{A}\cdots\eotimes{A}B}_{\text{$n+1$-times}})\\
    \cong & A\eotimes{A^e} (X\otimes\underbrace{C\otimes\cdots\otimes C}_{\text{$n$-times}})\\
    \cong & X/[X,A]\otimes (\underbrace{C\otimes\cdots\otimes C}_{\text{$n$-times}})\\
    \cong & (X/[X,A])\eotimes{C^e}\CB_*(C) 
  \end{align*}
  The result follows.
\end{proof}

\begin{theorem}\label{MainCJZ}
  Assume $A\subseteq B$ is a reduced flat Galois fibration with fibres in $C$.
  Then there are long exact sequences in Hochschild homology
  \begin{equation}
    \label{CyclicJZ}
    \cdots \to  B/[B,A]\otimes HH_{n+1}(C)\to HH_n(A) \to HH_n(B)\to B/[B,A]\otimes HH_n(C)\to \cdots 
  \end{equation} 
  for $n\geq 1$.
\end{theorem}

\begin{proof}
  Since the extension is reduced flat, it follows from \cite[Theorem
  4.2]{Kaygun:JacobiZariski} that we have the long exact sequence~\ref{JZ} for
  $n\geq 1$.  We also have that
  $$\CH_*(B|A)\cong B/[B,A]\otimes \CH_*(C)$$
  by Theorem~\ref{MainHJZ}.  The result follows.
\end{proof}

\section{Homology of Graph Coverings}\label{sect:GraphCovering}

Throughout this section, we assume $\C{G}$ is a groupoid with object set $V$.  We
also assume $G=(V,E)$ is a simple graph, i.e. $V$ is a set and $E$ is a subset of
multi-subsets of $V$ of size 2.

\subsection{Groupoids and their actions}

A covariant (resp. contravariant) functor $X\colon \C{G}\to \Set$ is called a left
(resp. right) $\C{G}$-set.  One can alternatively define a left $\C{G}$ set with
the following datum:
\begin{enumerate}[(1)]
\item There is a function $i\colon X\to V$,
\item For every $x\in i^{-1}(c)$ and $b\xla{g}c\in\C{G}$ there is an
  element $(b\xla{g}c)\cdot x\in i^{-1}(b)$ such that
  \begin{enumerate}[(i)]
  \item $(c\xla{id}c)\cdot x = x$, and
  \item $(a\xla{f}b)\cdot((b\xla{g}c)\cdot x) = (a\xla{fg}c)\cdot x$
    for every $a\xla{f}b\in\C{G}$.
  \end{enumerate}
\end{enumerate}

One can similarly define a bilateral $\C{G}$-set $M$ either as a functor of the
form $M\colon \C{G}^e\to \Set$ where $\C{G}^e$ is the enveloping groupoid
$\C{G}\times_V \C{G}^{op}$ with a sequence of conditions similar to the conditions
given above.

\subsection{Groupoids and distributive laws}

\begin{proposition}\label{Factorization}
  Assume we have two subgroupoids $\C{H}$ and $\C{K}$ in $\C{G}$ such that the
  composition in $\C{G}$ induces bijections of the form
  \begin{equation}
    \label{FactorizationSystem}
    \bigsqcup_{u\in V} \Hom_{\C{K}}(u,w)\times \Hom_{\C{H}}(v,u) 
    \xra{\circ} \Hom_{\C{G}}(v,w) \xla{\circ}
    \bigsqcup_{u\in V} \Hom_{\C{H}}(u,w)\times \Hom_{\C{K}}(v,u) 
  \end{equation}
  Then there is an invertible distributive law of the form
  $\omega\colon\C{K}\times_V \C{H} \to \C{H}\times_V \C{K}$.
\end{proposition}

\begin{proof}
  Let us denote the inverse of the right leg of the zig-zag of the bijections given
  in \eqref{FactorizationSystem} by $\xi$.  We define a morphism of groupoids
  $\omega\colon\C{K}\times_V \C{H} \to \C{H}\times_V \C{K}$ is the composition map
  followed by $\xi$
  \[ \C{K}\times_V \C{H} \xra{\circ} \C{G}\xra{\xi} \C{H}\times_V
  \C{K} \]
  We need to prove that this map is a left and a right transposition.
  We consider the diagram
  \[\xymatrix{
    \C{K}\times_V\C{K}\times_V\C{H} \ar[r]^{id\times\omega}\ar[d]_{\circ\times id} &
    \C{K}\times_V\C{H}\times_V\C{K} \ar[r]^{\omega\times id} &
    \C{H}\times_V\C{K}\times_V\C{K} \ar[d]^{id\times\circ} \\
    \C{K}\times_V\C{H} \ar[rr]^{\omega}\ar[dr]_{\circ} & &  \C{H}\times_V\C{K}\ar[dl]^{\circ} \\
    & \C{G}
  }\]
  Since the lower triangle is composed of compatible bijections, the upper rectangle
  must commute.
\end{proof}

\subsection{The free groupoid of a graph}

A path in $G$ is a finite sequence of vertices $(v_n,\ldots,v_0)$ such that
$\{v_{i+1},v_i\}\in E$ for every $i=0,\ldots,n-1$.  A path is called a cycle if it
starts and ends at the same vertex.  Let $\pi_1(G,x)$ be the free group generated
by all cycles on a vertex $x\in V$ subject to the relation
\begin{equation}\label{Reduction}
   (v,w,v) = v 
\end{equation}
for every edge $\{v,w\}\in E$. We now let $\pi_1(G)$ be the discrete groupoid
defined as the disjoint union $\bigsqcup_{v\in V} \pi_1(G,v)$, and we also define
$\C{F}_G$ to be the groupoid on $G$ where given any pair of vertices $x$ and $y$,
the Hom-set $\Hom_{\C{F}_G}(x,y)$ is the set of all paths from $x$ to $y$ subject
to the same relation as in Equation~\eqref{Reduction}.  Notice that in this
groupoid the inverse of an arrow (path) is the reverse path.

\subsection{The canonical groupoid of a graph}

Define a groupoid $\C{C}_G$ from $G$ as follows: the set of objects of $\C{C}_G$ is
the set of vertices $V$ of $G$.  For each $v,w\in V$ the set $\Hom_{\C{C}_G}(v,w)$
is empty if and only if there are no paths between $v$ and $w$ in $G$.  If there a
path, then the set $\Hom_{\C{C}_G}(v,w)$ contains a unique morphism simply denoted
by $w\xla{}v$, or by $p_{w,v}$ whenever it is convenient.  Since every arrow in
$\C{C}_G$ is invertible, this is a groupoid.

\begin{lemma}\label{FundamentalGroup}
  There is an invertible distributive law groupoids of the form
  \[ \bowtie\colon \pi_1(G)\times_V \C{C}_G\to \C{C}_G\times_V \pi_1(G) \] and an
  isomorphism of groupoids $\C{F}_G\cong \C{C}_G\bowtie\pi_1(G)$.
\end{lemma}

\begin{proof}
  Let us define an equivalence relation $\sim_1$ on the morphisms of $\C{F}_G$ as
  follows: For every $\alpha,\beta\in \Hom_{\C{F}_G}(x,y)$ we write
  \[ \alpha\sim_1\beta \text{ if and only if there is one } \gamma\in \pi_1(G,x)
    \text{ with } \alpha = \beta\gamma
  \] 
  We claim the quotient $\C{F}_G/_{\sim_1}$ is the canonical groupoid $\C{C}_G$.
  It is clear that the Hom-set $\Hom_{\C{F}_G/_{\sim_1}}(x,y)$ is the set of
  equivalence classes $\Hom_{\C{F}_G}(x,y)/_{\sim_1}$, and that each quotient set
  $\Hom_{\C{F}_G}(x,y)/_{\sim_1}$ either contains no elements, or contains exactly
  one element for every $x,y\in V$.  This follows from the fact that given any two
  paths $\alpha,\beta\in \Hom_{\C{F}_G}(x,y)$ we have $\alpha\sim_1\beta$ since
  $\alpha = \beta(\beta^{-1}\alpha)$.  So, set-wise the canonical groupoid and our
  quotient object have the same elements.  We must verify that we have an
  associative composition defined on the equivalence class of morphisms.  To this
  end, let us take $a,a'\in[\alpha]\in \Hom_{\C{F}_G/_{\sim_1}} (y,z)$ and
  $b,b'\in[\beta]\in \Hom_{\C{F}_G/_{\sim_1}} (x,y)$ with $a'=ac$ and $b'=bd$.  We
  see that $ab = ab'd^{-1}$ which means $ab\sim_1 ab'$.  On the other hand
  $ab' = a'c^{-1} b' = a'b' (cb')^{-1}b'$ which means $ab'\sim a'b'$.  Thus we see
  that the composition $[\alpha][\beta]\in \Hom_{\C{F}_G/_{\sim_1}}(x,z)$ is
  well-defined.  The fact that the composition is associative and has an identity
  follows immediately.  Moreover, we also have a bijection
  \[ \Hom_{\C{F}_G}(x,y) = \Hom_{\C{F}_G/_{\sim_1}}(x,y) \times \pi_1(G,x) \] for
  every $x,y\in V$.  On the other hand, we also have a dual equivalence relation
  $\sim_2$ defined as
  \[ \alpha\sim_2 \beta \text{ if and only if there is one } \gamma\in\pi_1(G,y)
    \text{ with } \gamma\alpha = \beta
  \]
  Notice that we have $\alpha\sim_1 \beta$ if and only if $\alpha\sim_2 \beta$ and
  we have a bijection of the form
  \[ \Hom_{\C{F}_G}(x,y) = \pi_1(G,y) \times \Hom_{\C{F}_G/_{\sim_2}}(x,y) \] for
  every $x,y\in V$.  The the result follows from Proposition~\ref{Factorization}.
\end{proof}

\subsection{Unramified covering of a graph}\label{Covering}

Our main reference for graph coverings is \cite{Friedman:SheavesOnGraphs}.

If $G=(V,E)$ and $G'=(V',E')$ are two simple graphs, a function $f\colon V'\to V$
is called \emph{a map of graphs} if the induced map on the edges restricts to a map
of the form $f\colon E'\to E$.  We are going to represent this category of graphs
as $\Graph$.

\begin{definition}\label{defn:covering}
  A map of simple graphs $f\colon G'\to G$ is called \emph{a covering} if
  $f\colon V'\to V$ is onto.  A covering $f$ is called \emph{finite} if $f^{-1}(v)$
  is a finite set for every $v\in V$.  A finite covering $f$ is called \emph{an
    $n$-fold covering} if $|f^{-1}(v)|=n$ for every $v\in V$.  A covering $f$ is
  called \emph{unramified} (resp. \emph{étale}) if for every $\{a,b\}\in E$ there
  is a bijective (resp. injective) function
  $\sigma_{b,a}\colon f^{-1}(a)\to f^{-1}(b)$ such that
  $\{x,\sigma_{b,a}(x)\}\in E'$ for every $x\in f^{-1}(a)$.
\end{definition}

For the rest of the subsection, assume $f\colon G'\to G$ is an unramified covering.

\begin{lemma}\label{LemmaUnramifiedCovering}
  The vertex set $V'$ of $G'$ is a bilateral $\C{F}_G$-set and there are bijections
  of the form $P(G)\times_V V' \cong P(G') \cong V'\times_V P(G) $.  Thus we have
  an invertible transposition between $\C{F}_G$ and $V'$.
\end{lemma}

\begin{proof}
  We have $f\colon V'\to V$ and $s\colon E\to V$.  Since the covering is étale,
  there is an injective map $\sigma_{b,a}\colon f^{-1}(a)\to f^{-1}(b)$ for every
  $(a,b)\in E$.  Then for every $(x,y)\in E'$ can be written as
  $(x,\sigma_{b,a}(x))$ where $a=f(x)$ and $b=f(y)$.  So, there is a bijection of
  the form $E'\cong V'\times_V E$.  One can extend the bijection to
  $P(G')\cong V'\times_V P(G)$.  The other bijection is obtained similarly.  Then
  the result follows from Proposition~\ref{Factorization}.
\end{proof}

\begin{proposition}\label{RelativeZeroDimensional}
  There is an isomorphism of groupoids of the form
  $\C{F}_{G'}\cong V'\times_V \C{F}_G$ where source and target maps for
  $V'\times_V \C{F}_G$ are defined as
  \[ s(x, f(x)\xra{\alpha} v) = x \quad\text{ and }\quad t(x, f(x)\xra{\alpha} v) =
    x\cdot\alpha \]
\end{proposition}

\begin{proof}
  Follows from Lemma~\ref{LemmaUnramifiedCovering}.
\end{proof}

\subsection{Fibrations of path algebras}\label{HochschildOfGraphAlgebras}

For the sake of brevity, we are going to use $\C{A}_{G}$ to denote the groupoid
algebra $\Bbbk[\C{F}_{G}]$ for every graph $G$.  Our main reference for Hochschild
and cyclic (co)homology of path algebras
is~\cite{Benson:CyclicHomologyOfPathAlgebras}.

The following result is well-known.  We furnish a proof for the sake of
completeness.
\begin{lemma}\label{CohomologicalDimensionOne}
  The Hochschild cohomological dimension of $\C{A}_G$ is at most 1 for every graph
  $G$.
\end{lemma}

\begin{proof}
  Let $F$ be the disjoint union of edges $E$ of $G$, and its inverse edges
  $E^{-1}$.  Then $\C{A}_G$ viewed as a bimodule over itself has a short resolution
  of the form
  \[ 0 \to \C{A}_G\eotimes{V} F\eotimes{V} \C{A}_G \xrightarrow{\delta} \C{A}_G\eotimes{V} \C{A}_G \to 0
  \]
  where
  \[ \delta(\alpha\otimes f\otimes \beta) = \alpha f\otimes\beta - \alpha\otimes f
    \beta \] for every homogeneous element $\alpha\otimes f\otimes \beta$ in
  $\C{A}_G\eotimes{V} F\eotimes{V} \C{A}_G $.
\end{proof}

For the rest of the subsection, assume $g\colon G'\to G$ is a finite unramified
connected covering of a connected graph $G$.

\begin{proposition}\label{GraphMainResult}
  There is a smooth Galois fibration of groupoid algebras of the form
  $\C{A}_G\subseteq \C{A}_{G'}$.
\end{proposition}

\begin{proof}
  We embed $\C{A}_G$ to $\C{A}_{G'}$ by sending each idempotent $e\in V$ to
  $\sum_{f\in g^{-1}(v)} f$ in $\C{A}_{G'}$.  This determines a unique map sending
  each edge $\{x,y\}\in E$ to an element in $\C{A}_{G'}$.  Smoothness is forced on
  the extension by Lemma~\ref{CohomologicalDimensionOne} since both algebras have
  Hochschild cohomological dimension at most 1.
\end{proof}

Observe that since $\C{A}_{G'}$ is free over $\C{A}_{G}$ it is faithfully flat, and
since the extension was smooth it is also reduced flat by
Theorem~\ref{AlmostSmoothBegetsReducedFlat}.

\begin{theorem}\label{UnramifiedCovering}
  The relative Hochschild homology groups $HH_n(\C{A}_{G'}| \C{A}_G,X)$ are trivial
  for every $\C{A}_{G'}$-bimodule $X$ and for $n\geq 1$.  Thus there is a natural
  epimorphism $HH_1(\C{A}_G,X)\to HH_1(\C{A}_{G'},X)$ for every
  $\C{A}_{G'}$-bimodule $X$, and isomorphism of the form
  $HC_n(\C{A}_G)\cong HC_n(\C{A}_{G'})$ for every $n\geq 2$.
\end{theorem}

\begin{proof}
  The result follows from Proposition~\ref{RelativeZeroDimensional},
  Lemma~\ref{CohomologicalDimensionOne}, Proposition~\ref{GraphMainResult},
  Proposition~\ref{HJZ}, and Theorem~\ref{MainCJZ}.
\end{proof}

There is an analogous isomorphism $HH_1(\C{A}_G)\cong HH_1(\C{A}_{G'})$ one can get
from \cite{Benson:CyclicHomologyOfPathAlgebras}, but the epimorphism in
Theorem~\ref{UnramifiedCovering} works with Hochschild homology with arbitrary
coefficients.  Moreover, since the isomorphisms in cyclic cohomology in
Theorem~\ref{UnramifiedCovering} are obtained by a fibration sequence, one can try
to get similar results for path algebras with relations provided we can write an
appropriate fibration, i.e. reduced flat extension.

\subsection{Local Coefficients on graphs and their homology}\label{subsect:LocalCoefficients}

Consider our definition of a covering of graphs $f\colon G'\to G$ we gave in
Definition~\ref{defn:covering}.  One can think of an unramified covering as a
groupoid whose set of objects is the set $V$ of vertices of the base $G$, and whose
morphisms are given by the structure bijections
$e_{y,x}\colon f^{-1}(x)\to f^{-1}(y)$.  Or, one can think of them as \emph{local
  coefficient systems:}
\begin{definition}
  A local coefficient system $\C{H}$ on a graph $G$ is a collection of objects
  $\{H_x\}_{x\in V}$ in a category (sets, groups, algebras, Hopf algebras etc.)
  together with a collection of isomorphisms $e_{x,y}\colon H_x\to H_y$ for every
  edge $(x,y)\in E$.
\end{definition}
One can see that the fundamental groupoid $\pi_1(G) := \{\pi_1(G,x)\}_{x\in V}$ is
a local coefficient system of groups for every graph $G$.  Also, every unramified
covering $G'\to G$ is a local coefficient system of sets on $G$, by definition.  We
refer the reader to~\cite{Steenrod:HomologyWithLocalCoefficients} or
\cite[pg. 58]{Spanier:AlgebraicTopology} for local coefficient systems defined on
topological spaces.

Let us assume $f\colon G'\to G$ is a local coefficient system of sets on $G$,
i.e. an unramified covering over $G$.  For every $v\in V$, let us define a
subgroupoid
\[ \C{S}_v = \{\alpha\in \pi_1(G,v)| \alpha\triangleright x = x, \text{ for every }
  x\in f^{-1}(v) \} \] It is easy to see that if $\beta$ is a path from a vertex
$v$ to another $w$, then $\beta \C{S}_v \beta^{-1} = \C{S}_w$.  Thus, $\C{S}$ is
another local coefficient system of groups on $G$, and more importantly, it is
normal in $\pi_1(G)$.  Now, we have another local coefficient system of groups
$\pi_1(G)/\C{S}$.  We call this new system as \emph{the monodromy groupoid} of
$G'$, and denote it by $\C{M}_{G'\downarrow G}$, where the monodromy group on each
vertex is denoted by $\C{M}_v := \pi_1(G,v)/S_v$ for every $v\in V$.  One can
easily see that we have
\begin{proposition}   
  The distributive law in Lemma~\ref{FundamentalGroup} we had for $\C{C}_G$ and
  $\pi_1(G)$ now extends to a distributive law between $\C{C}_G$ and $\C{S}$, and
  we get $\C{M}_{G'\downarrow G} \cong \C{C}_G\bowtie \pi_1(G)/\C{S}$.
\end{proposition}

Assume $f\colon G'\to G$ is a finite unramified connected covering of a connected
graph $G$.  By abuse of notation, let us use $\C{M}_{G'\downarrow G}$ to denote the
algebra $\Bbbk[\C{M}_{G'\downarrow G}]$ of the covering.
\begin{theorem}\label{LocalCoefficients}
  The extension $\C{C}_G\subseteq \C{M}_{G'\downarrow G}$ is reduced flat, and we
  have isomorphisms in Hochschild and cyclic homologies (written here only for
  Hochschild homology)
  \[ HH_n(\C{M}_{G'\downarrow G}) \cong HH_n(\C{M}_{G'\downarrow G}| \C{C}_G)\cong
    HH_n(\C{M}_v) \] for every $n\geq 1$ and $v\in V$.  And, since we assume
  $char(\Bbbk)=0$, we have $HC_n(\C{M}_{G'\downarrow G})\cong \Bbbk$ for all
  $n\geq 0$.
\end{theorem}

\begin{proof}
  We have an unramified smooth Galois fibration
  $\C{C}_G\subseteq \C{M}_{G'\downarrow G}$, and
  $\C{M}_{G'\downarrow G}/[\C{M}_{G'\downarrow G},\C{C}_G]$ is $\Bbbk$ since the
  cover is assumed to be connected.  Thus the extension is also reduced flat by
  Theorem~\ref{AlmostSmoothBegetsReducedFlat}.  Now, we use Theorem~\ref{MainCJZ}.
  For the second assertion, we observe that the group (co)homology $H_*(G,\Bbbk)$
  of a finite group $G$ over a field $\Bbbk$ of characteristic 0 is trivial for
  every $n\geq 1$.  Then we use~\cite{Burghelea:CyclicHomologyOfGroupRings}.
\end{proof}

\begin{example}
  Let $C_n$ be the \emph{the cycle graph} on $n$-vertices for $n\geq 3$.  Then for
  every $k\geq 2$, there is a unique connected $k$-cover
  $f_{k,n}\colon C_{kn}\to C_n$.  Figure~\ref{2-3cover} is a depiction of the only
  connected 2-covering $f_{2,3}\colon C_6\to C_3$.
  \begin{figure}[h]
    \centering
      \[\xymatrix{
      \bullet \ar@{..}[dd]\ar@{-}[dr] \ar@{-}[rr] & & \bullet \ar@{-}[dddl]\ar@{..}[dd]\\
      & \bullet \ar@{..}[dd]& \\
      \bullet \ar@{..>}[dd]\ar@{-}[dr] \ar@{-}[rr] & & \bullet \ar@{-}[ul]\ar@{..>}[dd]\\
      & \bullet \ar@{..>}[dd] & \\
      \bullet \ar@{-}[dr] & & \bullet \ar@{-}[ll]\\
      & \bullet \ar@{-}[ur] &  }\]
  \caption{$f_{2,3}\colon C_6\to C_3$, the only connected 2-covering of $C_3$.}
    \label{2-3cover}
  \end{figure}
  On each vertex in $C_n$, the fundamental group is $\B{Z}$.  The stabilizer group
  $\C{S}_v$ of each vertex $v$ in $C_{kn}$ over $C_n$ is $k$-times the generator in
  $C_n$.  So, the monodromy group is $\B{Z}/k$.  Then the relative homologies are
  calculated by Theorem~\ref{LocalCoefficients}.
\end{example}


\begin{thebibliography}{10}

\bibitem{Andre:AQCohomology}
M.~Andr\'e.
\newblock {\em Homologie des alg\`ebres commutatives}, volume 206 of {\em Die
  Grundlehren der mathematischen Wissenschaften}.
\newblock Springer, 1974.

\bibitem{Beck:DistributiveLaws}
J.~Beck.
\newblock {\em Distributive Laws}, volume~80 of {\em Lecture Notes in
  Mathematics}.
\newblock AMS, 1969.

\bibitem{Benson:CyclicHomologyOfPathAlgebras}
D.J. Benson.
\newblock Cyclic homology and path algebra resolutions.
\newblock {\em Math. Proc. Cambridge Philos. Soc.}, 106(1):57--66, 1989.

\bibitem{BrzezinskiWisbauer:ComodulesCorings}
T.~Brzezinski and R.~Wisbauer.
\newblock {\em Corings and comodules}, volume 309 of {\em London Mathematical
  Society Lecture Note Series}.
\newblock Cambridge University Press, Cambridge, 2003.

\bibitem{Burghelea:CyclicHomologyOfGroupRings}
D.~Burghelea.
\newblock The cyclic homology of the group rings.
\newblock {\em Commen. Math. Helv.}, 60(3), 1985.

\bibitem{CapEtAl:TwistedTensorProduct}
A.~Cap, H.~Schichl, and J.~Van\v{z}ura.
\newblock On twisted tensor products of algebras.
\newblock {\em Comm. Algebra}, 23(12):4701--4735, 1995.

\bibitem{Connes:Book}
A.~Connes.
\newblock {\em Noncommutative geometry}.
\newblock Academic Press Inc., San Diego, CA, 1994.

\bibitem{Friedman:SheavesOnGraphs}
J.~Friedman.
\newblock Sheaves on graphs, their homological invariants, and a proof of the
  {H}anna {N}eumann conjecture: with an appendix by {W}arren {D}icks.
\newblock {\em Mem. Amer. Math. Soc.}, 233(1100):xii+106, 2015.

\bibitem{Grothendieck:DescentI}
A.~Grothendieck.
\newblock Technique de descente et th\'eor\`emes d'existence en g\'eom\'etrie
  alg\'ebrique. {I}. {G}\'en\'eralit\'es. {D}escente par morphismes
  fid\`element plats.
\newblock In {\em S\'eminaire {B}ourbaki, {V}ol.\ 5}, pages Exp.\ No.\ 190,
  299--327. Soc. Math. France, Paris, 1995.

\bibitem{Johnson:CohomologyInBanachAlgebras}
B.~E. Johnson.
\newblock {\em Cohomology in {B}anach algebras}.
\newblock American Mathematical Society, Providence, R.I., 1972.
\newblock Memoirs of the American Mathematical Society, No. 127.

\bibitem{Kaygun:UniversalHopfCyclicTheory}
A.~Kaygun.
\newblock The universal {H}opf cyclic theory.
\newblock {\em Journal of Noncommutative Geometry}, 2(3):333--351, 2008.

\bibitem{Kaygun:JacobiZariski}
A.~Kaygun.
\newblock {J}acobi-{Z}ariski exact sequence for {H}ochschild homology and
  cyclic (co)homology.
\newblock {\em Homology, Homotopy and Applications}, 14(1):65--78., 2012.

\bibitem{KnusOjanguren:Descent}
M.-A. Knus and M.~Ojanguren.
\newblock {\em Th\'eorie de la descente et alg\`ebres d'{A}zumaya}.
\newblock Lecture Notes in Mathematics, Vol. 389. Springer-Verlag, Berlin-New
  York, 1974.

\bibitem{Lazard:Flatness}
D.~Lazard.
\newblock Sur les modules plats.
\newblock {\em C. R. Acad. Sci. Paris}, 258:6313--6316, 1964.

\bibitem{Loday:CyclicHomology}
J.-L. Loday.
\newblock {\em Cyclic homology}, volume 301 of {\em Die Grundlehren der
  Mathematischen Wissenschaften}.
\newblock Springer--Verlag, Berlin, second edition, 1998.

\bibitem{Milne:EtaleCohomology}
J.~S. Milne.
\newblock {\em \'Etale cohomology}, volume~33 of {\em Princeton Mathematical
  Series}.
\newblock Princeton University Press, Princeton, N.J., 1980.

\bibitem{Quillen:AQCohomology}
D.~Quillen.
\newblock On the (co-)homology of commutative rings.
\newblock In {\em Applications of Categorical Algebra (Proc. Sympos. Pure
  Math., Vol. XVII, New York, 1968)}, pages 65--87. Amer. Math. Soc.,
  Providence, R.I., 1970.

\bibitem{RosenbergZelinsky:AmitsurComplex}
A.~Rosenberg and D.~Zelinsky.
\newblock On {A}mitsur's complex.
\newblock {\em Trans. Amer. Math. Soc.}, 97:327--356, 1960.

\bibitem{Schelter:SmoothAlgebras}
W.~F. Schelter.
\newblock Smooth algebras.
\newblock {\em J. Algebra}, 103(2):677--685, 1986.

\bibitem{Schneider:PrincipalHomogeneousSpaces}
H.-J. Schneider.
\newblock Principal homogeneous spaces for arbitrary {H}opf algebras.
\newblock {\em Israel J. Math.}, 72(1-2):167--195, 1990.

\bibitem{Spanier:AlgebraicTopology}
E.~H. Spanier.
\newblock {\em Algebraic topology}.
\newblock McGraw-Hill Book Co., New York-Toronto, Ont.-London, 1966.

\bibitem{Steenrod:HomologyWithLocalCoefficients}
N.~E. Steenrod.
\newblock Homology with local coefficients.
\newblock {\em Ann. of Math. (2)}, 44:610--627, 1943.

\bibitem{WeibelGeller:EtaleDescent}
C.~A. Weibel and S.~C. Geller.
\newblock \'{E}tale descent for {H}ochschild and cyclic homology.
\newblock {\em Comment. Math. Helv.}, 66(3):368--388, 1991.

\end{thebibliography}

\end{document}